\documentclass[12pt,centertags,oneside]{amsart}

\usepackage{amsmath,amstext,amsthm,amscd,typearea,hyperref,stmaryrd, xcolor}
\usepackage{amssymb}
\usepackage{a4wide}
\usepackage[mathscr]{eucal}
\usepackage{mathrsfs}
\usepackage{typearea}
\usepackage{charter}
\usepackage{pdfsync}
\usepackage[a4paper,width=16.2cm,top=3cm,bottom=3cm]{geometry}

\numberwithin{equation}{section}

\newtheorem{theorem}{Theorem}[section]
\newtheorem{definition}[theorem]{Definition}
\newtheorem{proposition}[theorem]{Proposition}
\newtheorem{corollary}[theorem]{Corollary}
\newtheorem{lemma}[theorem]{Lemma}
\newtheorem{remark}[theorem]{Remark}

\newcommand{\cali}[1]{\mathscr{#1}}

\newcommand{\Aut}{{\rm Aut}}

\newcommand{\supp}{{\rm supp}}

\newcommand{\dist}{{\rm dist}}

\newcommand{\DSH}{{\rm DSH}}

\newcommand{\id}{{\rm id}}

\newcommand{\Cc}{\cali{C}}

\newcommand{\Uc}{\cali{U}}
\newcommand{\Vc}{\cali{V}}

\newcommand{\C}{\mathbb{C}}

\newcommand{\N}{\mathbb{N}}

\renewcommand{\P}{\mathbb{P}}

\title[]{An equidistribution theorem for biraitonal maps of $\P^k$}

\author{Taeyong Ahn}
\address{(Ahn) Department of Mathematics Education, Inha University, 100 Inha-ro, Michuhol-gu, Incheon 22212, Republic of Korea}%
\email{t.ahn@inha.ac.kr}

\date{\today}

\begin{document}
\begin{abstract}
We prove an equidistribution theorem of positive closed currents for a certain class of birational maps $f_+:\P^k\to\P^k$ of algebraic degree $d\geq 2$ satisfying $\bigcup_{n\geq 0}f_-^n(I^+)\cap \bigcup_{n\geq 0}f_+^n(I^-)=\emptyset$, where $f_-$ is the inverse of $f_+$ and $I^\pm$ are the sets of indeterminacy for $f_\pm$,  respectively.
\end{abstract}

%\begin{abstract}
%In this note, we consider a birational map $f_+:\P^k\to\P^k$ of algebraic degree $d$. Let $f_-$ denote its inverse, $I^\pm$ the sets of indeterminacy of $f_\pm$ and $I^\pm_\infty:=\bigcup_{n\geq 0}f_\mp^n(I^\pm)$, respectively. Suppose that $I^+_\infty\cap I^-_\infty=\emptyset$ and that there exists an open neighborhood $V$ of $I^-_\infty$ such that $f_+(V)\Subset V$. Let $s$ be a positive integer such that $\dim I^+=k-s-1$ and $\dim I^-=s-1$. Then, for every positive closed $(s, s)$-current $S$ whose super-potential is continuous in $V$, $d^{-sn}(f_+^n)^*S$ converges to a constant multiple of the Green $(s, s)$-current. The constant is determined by the mass of $S$.
%\end{abstract}
\maketitle

\section{Introduction}
Let $\omega$ be a Fubini-Study form chosen so that $\int_{\P^k}\omega^k=1$. For an integer $1\leq p\leq k$, $\Cc_p$ denotes the space of positive closed $(p, p)$-currents of unit mass on $\P^k$ where the mass is defined by $\|S\|:=\langle S, \omega^{k-p}\rangle$ for positive closed $(p, p)$ current $S$ on $\P^k$. For an open subset $W\subseteq \P^k$, $\Cc_p(W)$ is the set of currents $S\in\Cc_p$ with $\supp S\Subset W$.
\medskip

In \cite{Ahn18}, the following equidistribution theorem for regular polynomial automorphisms of $\C^k$ was proved.
\begin{theorem}
	[Theorem 1.3 in \cite{Ahn18}, See also \cite{Ahn16}, \cite{DS09}]\label{thm:regpoly}
	Let $f:\C^k\to\C^k$ be a regular polynomial automorphism of $\C^k$ of degree $d\geq 2$ and $s>0$ an integer such that $\dim I^+=k-s-1$ and $\dim I^-=s-1$ where $I^\pm$ are the sets of indeterminacy of $f, f^{-1}$, respectively. Then, for an integer $0<p\leq s$, for $S\in\Cc_p$ whose super-potential $\Uc_S$ of mean $0$ is continuous near $I^-$, $d^{-pn}(f^n)^*S$ converges to the Green $(p, p)$-current $T_+^p$ for $f$ in the sense of currents where $T^p_+=\lim_{n\to\infty}d^{-pn}(f^n)^*\omega^p$.
\end{theorem}
(For the notion of the local continuity and H\"older continuity of super-potentials, see Section \ref{sec:sup-pot}.) The motivation of this note is to futher study Theorem \ref{thm:regpoly} in the case of certain birational maps of $\P^k$.
\medskip

Let $f_+:\P^k\to \P^k$ be a birational map of algebraic degree $d\geq 2$ and $f_-$ its inverse. Let $\delta$ denote the algebraic degree of $f_-$. Let $I^\pm$ denote the indeterminacy sets of $f_{\pm}$ and $I^\pm_\infty:=\bigcup_{n\geq 0}f_\mp^n(I^\pm)$, respectively. Let $s>0$ be an integer such that $\dim I^+=k-s-1$ and $\dim I^-=s-1$. The main theorem of this note is as follows: 
\begin{theorem}
	\label{thm:main} Let $f_+:\P^k\to \P^k$ be a birational map of algebraic degree $d\geq 2$ such that $I^+_\infty\cap I^-_\infty=\emptyset$ and that there exists an open subset $V\subset\P^k$ such that $V\cap I^+_\infty=\emptyset$ and $I^-_\infty\subset f_+(V)\Subset V$. Then, for an integer $0<p\leq s$, for every $S\in\Cc_p$ whose super-potential is H\"older continuous in $V$, then we have $d^{-pn}(f_+^n)^*S$ converges to $T_+^p$ in the sense of currents where $T^p_+=\lim_{n\to\infty}d^{-pn}(f_+^n)^*\omega^p$.
\end{theorem}
\medskip

If we assume $I^-_\infty$ is attracting for $f_+$, we obtain
\begin{theorem}\label{thm:cor}
	Assume the hypotheses in Theorem \ref{thm:main} and further that $I^-_\infty$ is attracting for $f_+$. Then, for an integer $0<p\leq s$, for a generic analytic subset $H$ of pure dimension $k-p$ which means $H\cap I^-_\infty=\emptyset$, we have
	\begin{displaymath}
		d^{-pn}(f_+^n)^*[H] \to cT_+^p
	\end{displaymath}
	in the sense of currents where $c$ is the degree of $H$.
\end{theorem}
Here, a compact subset $A$ of $\P^k$ is called an attracting set if it has an open neighborhood $U$, called a trapping neighborhood, such that $f_+(U)\subseteq U$ and $A =\bigcup_{n\geq 0} f_+^n(U)$ where $f_+^n:=f_+\circ\cdots\circ f_+$, $n$-times.

Among numerous studies on the birational maps on $\P^k$, listing some works related to equidistribution of inverse images of positive closed currents, in \cite{Diller}, Diller proved that in $\P^2$, the equidistribution is true for $S\in\Cc_1$ with $\supp S\cap I^-_\infty=\emptyset$. In \cite{DS05}, Dinh-Sibony defined notions of regular birational maps and $PC_p(V)$-currents for an open subset $V$ of $\P^k$, which is equivalent to its super-potential $\Uc_S$ being continuous in $\P^k\setminus \overline{V}$ and proved that if the initial current $S\in\Cc_p$ satisfies a regularity condition in terms of $PC_p(V)$, then the equidistribution is true in a certain open subset of $\P^k$ where $V$ is an open neighborhood of $I^+_\infty$. In \cite{DV}, De Th\'elin-Vigny prove that for $f_+$ with $I^+_\infty\cap I^-_\infty=\emptyset$, outside a super-polar set of $\Cc_s$, equidistribution in Theorem \ref{thm:main} holds. They gave a sufficient condition in terms of super-potentials for the equidistribution. The condition in \cite{DV} is not stated in terms of the set $I^-_\infty$. In this note we focus on a sufficient condition in terms of the set $I^-_\infty$ of critical values of $f_+$ as in \cite{Diller} and equidistribution on the whole $\P^k$.
\medskip

The condition $I^+_\infty\cap I^-_\infty=\emptyset$ was introduced in \cite{Diller} and \cite{DV}. From the dynamical view point, that is, considering iteration of $f_\pm$, it seems reasonable to regard $I^-_\infty$ for birational maps as a generalization of $I^-$ in Theorem \ref{thm:regpoly} rather than $I^-$ alone. Along the same lines, the regularity condition in Theorem \ref{thm:regpoly} may be translated into $I^+_\infty\cap I^-_\infty=\emptyset$ for birational maps. If we compare the class of birational maps in this note, in \cite{DV} and in \cite{DS05}, ours contains the case of \cite{DS05} and ours belongs to the case of \cite{DV}.
\medskip

For the proof, we basically follow and refine the proof of Theorem 1.1 and 1.4 in \cite{Ahn18}. The main difficulty is to get uniform estimates of $\int_W S\wedge U_{\Lambda^n(R)}$ with respect to $n\in \N$ in a neighborhood $W$ of $I^-_\infty$ for smooth $R\in\Cc_{k-p+1}$ where $U_{(\cdot)}$ denotes the Green quasi-potential of a given current and $\Lambda$ is a constant multiple of the operator $(f_+)_*$. For this, we use the idea in the proof of Theorem 1.1 in \cite{Ahn18}, which essentially means that a Green quasi-potential can be approximated from below by a negative closed current with a small error. (See the proof of Proposition 2.3.6 in \cite{DS09}.)  Also, there are subtle differences between Theorem \ref{thm:main} and the case of regular polynomial automorphisms of $\C^k$. Firstly, we do not know whether for every $U\Subset \P^k\setminus I^+_\infty$, a super-potential of $T_+^p$ is continuous in $U$. This is needed to bound the dynamical super-potentials from above. As a replacement for this, we will use a convergence appearing in a proof of Theorem 3.2.4 in \cite{DV}. Next, in general, $I^-$ may not be invariant under $f_+$ and $I^-_\infty$ may not be an attracting set. For the former part, we construct another invariant current for $f_+$ which is denoted by $R_\infty$ in Proposition \ref{prop:invariant} and for the latter part, the assumption of the existence of the neighborhood $V$ such that $V\cap I^+_\infty=\emptyset$ and $I^-_\infty\subset f_+(V)\Subset V$ resolves the difficulty.
\smallskip

In this note, the $C^\alpha$-norm $\|\cdot\|_{C^\alpha}$, the uniform norm $\|\cdot\|$, the uniform norm $\|\cdot\|_U$ on a set $U$ are computed in terms of the sum of the coefficients of a given form with respect to a fixed finite atlas.

\section{Currents}\label{sec:currents}
In this note, we assume some familiarity of the reader to pluripotential theory and currents. For details, consult \cite{Demailly} and \cite{Sib} for instance. In this section, we introduce some notions and notations that we will use in this note.
\medskip

Let $1\leq q\leq k$ and $W$ an open subset of $\P^k$. The following spaces and norms are useful in the study of currents. For instance, see \cite{DS10}, \cite{DNV}, \cite{Ahn18}. Let $D_q$ be the real vector space spanned by $\Cc_q$ and $D^0_q(W)$ the subspace of $D_q$ of currents $R$ which are cohomologous to $0$ and satisfy $\supp R\Subset W$. We define $\|R\|_*:=\inf\{\|R_+\|: R=R_+-R_-, R_\pm \textrm{ positive and closed}\}$ on $D^0_q(W)$. Let $\widetilde D^0_q(W):=\{R\in D^0_q(W): \|R\|_*\leq 1\}$. The topology on $\widetilde D^0_q(W)$ is the subspace topology of the space of currents in $W$. Note that the space $\widetilde D^0_q(W)$ is compact. The norm $\|\cdot\|_*$ bounds the mass norm. So, $\widetilde D_q^0(W)$ is metrizable. More precisely, if $\gamma>0$ is a constant, we define for $R\in \widetilde D_q^0(W)$
\begin{displaymath}
	\|R\|_{-\gamma}:=\sup\{|\langle R, \phi\rangle|, \phi \textrm{ is a test form of bi-degree }(k-q, k-q)\textrm{ with }\|\phi\|_{C^\gamma}\leq 1 \}.
\end{displaymath}

In a similar fashion, we have
\begin{definition}[See \cite{DS06}]
	Let $\phi:\P^k\to\P^k$ be an $L^1$-function. We say that $\phi$ is a DSH function if outside a pluripolar set, $\phi$ can be written as a difference of two quasi-plurisubharmonic functions. Two DSH functions are identified if they are equal to each other outside a pluripolar set.

	If $\phi$ is a DSH function on $\P^k$, we define the DSH-norm of $\phi$ by
	\begin{displaymath}
		\|\phi\|_\DSH:=\|\phi\|_{L^1}+\|dd^c\phi\|_*.
	\end{displaymath}
\end{definition}
\bigskip

On $\P^k$, we have a good smooth approximation of positive closed currents. The following is from \cite{DS09}. We will simply call it the standard regularization or the $\theta$-regularization of a current. Since $\Aut(\P^k)\simeq \mathrm{PGL}(k+1, \C)$, we choose and fix a holomorphic chart such that $|y|<2$ and $y=0$ at $\id\in\Aut(\P^k)$. We denote by $\tau_y$ the automorphism corresponding to $y$. We choose a norm $|y|$ of $y$ so that it is invariant under the involution $\tau\to\tau^{-1}$. Fix a smooth probability measure $\rho$ with compact support in $\{y:|y|<1\}$ such that $\rho$ is radial and decreasing as $|y|$ increases. Then, the involution $\tau\to\tau^{-1}$ preserves $\rho$. Let $h_\theta(y):=\theta y$ denote the multiplication by $\theta\in\C$ and for $|\theta|\leq 1$ define $\rho_\theta:=(h_\theta)_*\rho$. Then, $\rho_0$ becomes the Dirac mass at $\id\in\Aut(\P^k)$. We define for $R\in\Cc_q$,
\begin{displaymath}
	R_\theta:=\int_{\Aut(\P^k)}(\tau_y)_*Rd\rho_\theta(y)=\int_{\Aut(\P^k)}(\tau_{\theta y})_*Rd\rho(y)=\int_{\Aut(\P^k)}(\tau_{\theta y})^*Rd\rho(y).
\end{displaymath}
Note that $R_\theta\in\Cc_q$.

\begin{proposition}
	[Proposition 2.1.6 in \cite{DS09}] If $\theta\neq 0$, then $R_\theta\in \Cc_q$ is a smooth form which depends continuously on $R$. Moreover, for every $\alpha\geq 0$ there is a constant $c_\alpha$ independent of $R$ such that
	\begin{displaymath}
		\|R_\theta\|_{C^\alpha}\leq c_\alpha\|R\||\theta|^{-2k^2-4k-\alpha}.
	\end{displaymath}
\end{proposition}

\section{Super-potentials}\label{sec:sup-pot}

For the details of super-potentials on $\P^k$, we refer the reader to \cite{DS09}. For the reader's convenience, we summarize some definitions and properties of super-potentials on $\P^k$.

\begin{definition}
	Let $0<q\leq k$ be an integer. For smooth $S\in\Cc_q$, the super-potential $\Uc_S$ of $S$ of mean $0$ is a function defined on $\Cc_{k-q+1}$ by
	\begin{displaymath}
		\Uc_S(R)=\langle U_S, R\rangle
	\end{displaymath}
	where $R\in\Cc_{k-q+1}$ and $U_S$ is a quasi-potential of $S$ of mean $0$, which is a $(q-1, q-1)$-current such that $S-\omega^s=dd^cU_S$ and $\langle U_S, \omega^{k-q+1}\rangle=0$.\medskip
	
	For a general current $S\in\Cc_q$,
	\begin{displaymath}
		\Uc_S(R)=\lim_{\theta\to 0} \Uc_{S_\theta}(R)
	\end{displaymath}
	where $S_\theta$ is the standard regularization of $S$ as in Section \ref{sec:currents} and $\Uc_{S_\theta}$ is its super-potential of $S_\theta$ of mean $0$.
\end{definition}

Among various quasi-potentials, there is a good one for a computational purpose. It is called the Green quasi-potential and given by an integral formula.
\begin{proposition}
	[Proposition 2.3.2 in \cite{DS09}]\label{prop:greenk} Let $\Delta$ be the diagonal submanifold of $\P^k\times\P^k$ and $\Omega$ a closed real smooth $(k, k)$-form cohomologous to $[\Delta]$. Then, there is a negative $(k-1, k-1)$-form $K$ on $\P^k\times\P^k$ smooth outside $\Delta$ such that $dd^cK=[\Delta]-\Omega$ which satisfies the following inequality near $\Delta$:
	\begin{displaymath}
		\|K(\cdot)\|_\infty\lesssim -\dist(\cdot, \Delta)^{2-2k}\log \dist (\cdot, \Delta)\quad \textrm{ and }\quad \|\nabla K(\cdot)\|_\infty \lesssim \dist(\cdot, \Delta)^{1-2k}.
	\end{displaymath}
	Moreover, there is a negative dsh function $\eta$ and a positive closed $(k-1, k-1)$-form $\Theta$ smooth outside $\Delta$ such that $K\geq \eta\Theta$, $\|\Theta(\cdot)\|_\infty\lesssim \dist(\cdot, \Delta)^{2-2k}$ and $\eta-\log\dist (\cdot, \Delta)$ is bounded near $\Delta$.
\end{proposition}
Here, the inequalities are up to a constant multiple independent of the point in $\P^k\times\P^k\setminus \Delta$. The norm $\|\nabla K\|_\infty$ is the sum $\sum_j |\nabla K_j|$, where the $K_j$'s are the coefficients of $K$ for a fixed atlas of $\P^k\times\P^k$.

We consider a fixed kernel $K$ throughout the rest of the note. The Green quasi-potential $U_S$ of $S$ is defined by
\begin{displaymath}
	U_S(z):=\int_{\zeta\neq z} K(z, \zeta)\wedge S(\zeta).
\end{displaymath}

Using the notion of super-potentials, we can define the operator $f_+^*$ on $\Cc_q$ where $f_+$ is a birational map in Theorem \ref{thm:main}.
\begin{definition}
	[Definition 5.1.4 in \cite{DS09}]\label{def:admissible} We say that $S\in\Cc_q$ is $f_+^*$-admissible if there is a current $R_0\in\Cc_{k-q+1}$ which is smooth on a neighborhood of $I^+$, such that the super-potential of $S$ are finite at $\Lambda_{k-q+1}(R_0)$.
\end{definition}

\begin{proposition}
	[Proposition 5.1.8 in \cite{DS09}] Let $S$ be an $f_+^*$-admissible current in $\Cc_p$. Let $\Uc_S$ and $\Uc_{L(\omega^p)}$ be super-potentials of $S$ and $L_p(\omega^p)$. Then, we have
	\begin{displaymath}
		\lambda_p(f_+)^{-1}\lambda_{p-1}(f_+)\Uc_S\circ\Lambda_{k-p+1} +\Uc_{L_p(\omega^p)}
	\end{displaymath}
	is equal to a super-potential of $L_p(S)$ for $R\in \Cc_{k-p+1}$, smooth in a neighborhood of $I^+$.
\end{proposition}

\begin{lemma}
	[Lemma 3.1.5 in \cite{DV}]\label{lem:iteration} Let $S\in \Cc_q$ for $0<q\leq k$. Let $n>0$ be such that $S$ is $(f_+^n)^*$-admissible then for all $j$ with $0\leq j\leq n-1$, $L^j(S)$ is well defined, $f_+^*$-admissible and $L^{j+1}(S)=L_{j+1}(S)$. In particluar, $L^n(S)=L_n(S)$.
\end{lemma}

%\begin{proposition}
%	[Proposition 5.1.8 in \cite{DS09}] Let $S$ be an $f_+^*$-admissible current in $\Cc_p$. Let $\Uc_S$ and $\Uc_{L(\omega^p)}$ be super-potentials of $S$ and $L_p(\omega^p)$. Let $S_n$ be currents in $\Cc_p$ $H$-converging to $S$. Then, $S_n$ are $f^*$-admissible and $f^*(S_n)$ H-converge towards $f^*(S)$. Moreover,
%	\begin{displaymath}
%	\lambda_p^{-1}\lambda_{p-1}\Uc_S\circ\Lambda +\Uc_{L(\omega^p)}
%	\end{displaymath}
%	is equal to a super-potential of $L(S)$ for $R\in \Cc_{k-p+1}$, smooth in a neighborhood of $I$.
%\end{proposition}

\section{Locally regularity of super-potentials}
In \cite{Ahn18}, the notions of locally bounded/continuous superpotentials were given as below. Similarly, we define local H\"older continuity. The notion of the H\"older continuity of super-potentials was given in \cite{DNV}. We will write $\Uc_S$ for the super-potential of a current $S\in\Cc_q$ of mean $0$.
\begin{definition}\label{def:local_cbh}
	Let $1\leq q\leq k$. Let $S\in\Cc_q$ and $W$ an open subset of $\P^k$. 
	The super-potential $\Uc_S$ of $S$ of mean $m$ is said to be bounded in $W$ if there exists a constant $C_S>0$ such that for any smooth current $R\in \widetilde D^0_{k-q+1}(W)$, we have
	\begin{displaymath}
		|\Uc_S(R)|\leq C_S.
	\end{displaymath}
	
	The super-potential $\Uc_S$ of $S$ of mean $m$ is said to be continuous in $W$ if $\Uc_S$ continuously extends to $\widetilde D^0_{k-q+1}(W)$ with respect to the subspace topology of the space of currents in $W$.
	\medskip
	
	The super-potential $\Uc_S$ of $S$ of mean $m$ is said to be H\"older continuous in $W$ if $\Uc_S$ continuously extends to $\widetilde D^0_{k-q+1}(W)$ and H\"older continuous with respect to one of the norms $\|\cdot\|_{-\gamma}$ on $\widetilde D^0_{k-q+1}(W)$.
\end{definition}
Note that when $W=\P^k$, our notion coincides with the definition in \cite{DNV}. Since $\widetilde D^0_{k-q+1}(W)$ is compact, $\Uc_S$ is continuous in an open subset $W\subset\P^k$, then it is bounded in an open subset $W\subset\P^k$.
\begin{remark}
	By interpolation theory, for any $\gamma\geq\gamma'>0$, there is a constant $c>0$ such that $\|\cdot\|_{-\gamma}\leq \|\cdot\|_{-\gamma'}\leq c(\|\cdot\|_{-\gamma})^{\gamma'/\gamma}$. So, if $\Uc_S$ is H\"older continuous for one $\|\cdot\|_{-\gamma}$, then it is H\"older continuous for all $\|\cdot\|_{-\gamma}$.
\end{remark}

\begin{remark}
	For $S\in\Cc_q$, its super-potential $\Uc_S$ is continuous in an open subset $W\subset \P^k$ if and only if $S$ is $PC_q(\P^k\setminus \overline W)$. The equivalence can be observed via the Green quasi-potential kernel in Proposition \ref{prop:greenk}.
\end{remark}

\begin{proposition}\label{prop:bounded_n_finite}
	If a super-potential $\Uc_S$ of $S\in\Cc_q$ is bounded in an open subset $W\subset \P^k$ and if $R\in\Cc_{k-q+1}$ is smooth outside a compact subset $K\Subset W$, then $\Uc_S(R)$ is finite.
\end{proposition}

\begin{proof}
	Let $\chi:\P^k\to [0,1]$ be a smooth cut-off function such that $\supp \chi \Subset W$ and $\chi\equiv 1$ on $K$. Then, $dd^c((1-\chi)U_R)$ is a smooth $(k-q+1, k-q+1)$-current and $dd^c(\chi U_R)$ is a current in $D_{k-q+1}^0(W)$. Hence,
	\begin{displaymath}
		\Uc_S(R)=\Uc_S(dd^c((1-\chi)U_R))+\Uc_S(dd^c(\chi U_R))+\Uc_S(\omega^{k-p+1})
	\end{displaymath}
	and so, it is finite as desired.
\end{proof}

\section{Birational Maps}
In this section, we summarize well-known properties of birational maps on $\P^k$. For details, see \cite{DV} for instance.
\medskip

Let $f_+:\P^k\to \P^k$ be a birational map of algebraic degree $d\geq 2$ and $f_-$ its inverse. Let $\delta$ denote the algebraic degree of $f_-$. Let $I^\pm$ denote the indeterminacy sets of $f_{\pm}$ and $I^\pm_\infty:=\bigcup_{n\geq 0}f_\mp^n(I^\pm)$, respectively. Let $s>0$ be an integer such that $\dim I^+=k-s-1$ and $\dim I^-=s-1$.
Let $C^\pm$ be the critical sets for $f^\pm$:
\begin{displaymath}
C^+:=f_+^{-1}(I^-)\quad{\rm and }\quad C^-:=f_-^{-1}(I^+).
\end{displaymath}
Then, we have $I^+\subset C^+$, $I^-\subset C^-$ and $f_+:\P^k\setminus C^+\to \P^k\setminus C^-$ is a biholomorphism (p.42 in \cite{DV}). We also have
\begin{align*}
f_-\circ f_+=\id\, \textrm{ on }\,\P^k\setminus C^+\quad&{\rm and }\quad f_+\circ f_-=\id \, \textrm{ on }\, \P^k\setminus C^-;\\
f_+(\P^k\setminus I^+)\subseteq(\P^k\setminus C^-)\cup I^-\quad&{\rm and }\quad f_-(\P^k\setminus I^-)\subseteq(\P^k\setminus C^+)\cup I^+.
\end{align*}
\medskip

For $0\leq q\leq k$ and $n>0$, we define $\lambda_q(f_+^n)$ by
\begin{displaymath}
\lambda_q(f_+^n):=\|(f_+^n)^*(\omega^q)\|=\|(f_+^n)_*(\omega^{k-q})\|.
\end{displaymath}
\begin{proposition}
	[Proposition 3.1.2 in \cite{DV}] We have $\lambda_q(f_+)=d^q$ for $q\leq s$ and $\lambda_q(f_+)=\delta^{k-q}$ for $q\geq s$. In particular, $d^s=\delta^{k-s}$.
\end{proposition}

\begin{proposition}
	[Corollary 3.1.4 in \cite{DV}]\label{prop:iteration} We have $(f_+^*)^n=(f_+^n)^*$ for smooth currents in $\Cc_q$ and $\lambda_q(f_+^n)=(\lambda_q(f_+))^n$ for all $0\leq q\leq k$.
\end{proposition}

We define two operators acting on $\Cc_q$:
\begin{displaymath}
	L_q:=(\lambda_q(f_+))^{-1}f_+^*\quad {\rm and}\quad \Lambda_q:=(\lambda_{k-q}(f_+))^{-1}(f_+)_*.
\end{displaymath}
Note that the operators $L_q$ and $\Lambda_q$ are well-defined for currents in $\Cc_q$ which are smooth near $I^-$ and $I^+$, respectively.
\medskip

\begin{proposition}\label{prop:f+f-smooth}
	Let $0<q\leq k$. Let $R$ be a smooth current of bidegree $(q, q)$. Then,
	\begin{displaymath}
		(f_+)_*R=(f_-)^*R
	\end{displaymath}
	as a current on $\P^k$ and $\supp (f_+)_*R=(f_-)^*R\subseteq \overline{(f_+)^{-1}(\supp R \setminus C_-)}=\overline{f_-(\supp R\setminus C_-)}$.
\end{proposition}
\begin{proof}
	Notice that $(f_+)_*R$ and $(f_-)^*R$ are both forms with $L^1$-coefficients. So, they do not charge any algebraic sets of dimension $\leq k-1$. Let $\varphi$ be a smooth test form of bidegree $(k-q, k-q)$. Then, we have
	\begin{align*}
		\langle (f_+)_*R, \varphi\rangle =\langle R, (f_+)^*\varphi\rangle=\langle R, (f_+)^*\varphi\rangle_{\P^k\setminus C^+}.
	\end{align*}
	Since $f_-:\P^k\setminus C^-\to\P^k\setminus C^+$ is a biholomorphism, the change of coordinates by $f_-$ implies
	\begin{align*}
		\langle R, (f_+)^*\varphi\rangle_{\P^k\setminus C^+}&=\langle (f_-)^*R, (f_-)^*((f_+)^*\varphi)\rangle_{\P^k\setminus C^-}\\
		&=\langle (f_-)^*R, \varphi\rangle_{\P^k\setminus C^-}=\langle (f_-)^*R, \varphi\rangle.
	\end{align*}
	The second last inequality is from $f_+\circ f_-:\P^k\setminus C^-\to \P^k\setminus C^-$ being identity on $\P^k\setminus C^-$.
	\medskip
	
	The support property is from direct computations together with the fact that the currents $(f_+)_*R$ and $(f_-)^*R$ have $L^1$-coefficients.
\end{proof}

\begin{corollary}\label{cor:f+f-smooth}
	For $R\in\Cc_{k-q+1}$ smooth outside $I^-_\infty$, $\Lambda_{k-q+1}^n (R)$ is smooth outside $I^-_\infty$.
\end{corollary}

Together with Lemma \ref{lem:iteration}, we obtain the following proposition:
\begin{proposition}
	If $S\in\Cc_q$ admits a super-potential bounded in a neighborhood of $I^-_\infty$, then for every $m, n\geq 0$, $L_q^n(S)$ is well defined and $L_q^{m+n}(S)=L_q^m(L_q^n(S))$.
\end{proposition}

\begin{proof}
	Let $R\in\Cc_{k-q+1}$ be a smooth current. Then, by Corollary \ref{cor:f+f-smooth} and Lemma \ref{lem:iteration}, $\Lambda_{k-q+1}^n(R)$ is well defined and smooth outside $I^-_\infty$. So, by Proposition \ref{prop:bounded_n_finite}, the super-potential of $S$ is finite at $\Lambda_{k-q+1}^n(R)$ for every $n$. Definition \ref{def:admissible} and Lemma \ref{lem:iteration} finish the proof.
\end{proof}

Now, we consider the Green current of order $q$ associated to $f_+$. We further assume the existence of an open subset $V\subset\P^k$ such that $V\cap I^+_\infty=\emptyset$ and $I^-_\infty\subset f_+(V)\Subset V$ as in Theorem \ref{thm:main}. Then, there is a strictly positive distance between $I^+_\infty$ and $I^-_\infty$. Then, this satisfies Hypothesis 3.1.6 in \cite{DV}. As a result of it, we have the existence of the Green current of order $q$ as below:
\begin{theorem}[Theorem 3.2.2 in \cite{DV}]
	Let $0<q\leq s$. The sequence $(L_q^m(\omega^q))$ converges in the Hartog's sense to the Green current $T^q_+$ of order $q$ of $f$. Further, $\Uc_{T_+^s}([I^-])>-\infty$.
\end{theorem}

The following proposition can be obtained via a slight modification of Lemma 5.4.2 and Lemma 5.4.3 in \cite{DS09}.
\begin{proposition}\label{prop:Holder}
	Suppose that there is an open subset $V\subset\P^k$ such that $V\cap I^+_\infty=\emptyset$ and $I^-_\infty\subset f_+(V)\Subset V$ as in Theorem \ref{thm:main}. Let $0<q\leq s$. Then, the Green current $T_+^q$ of order $q$ is H\"older continuous on $\Cc_{k-q+1}(V)$.
\end{proposition}
Here, the H\"older continuity is with respect to the same $\|\cdot\|_{-\gamma}$-norm ($\gamma >0$) as in Definition \ref{def:local_cbh}, but the difference is that we are taking a different set $\Cc_{k-q+1}(V)$ of currents other than $\widetilde{D}^0_{k-q+1}(V)$.
\\

In the rest of this section, we construct a $(k-p+1, k-p+1)$-current $R_\infty^p$ such that $\Lambda_{k-p+1}(R_\infty^p)=R_\infty^p$ and $\supp R_\infty^p \subset \overline{I^-_\infty}$ where $\Lambda_{k-s+1}=d^{-(p-1)}(f_+)_*$ for $0<p\leq s$.
\medskip

We will first construct such a current $R_\infty^s$ of bidegree $(k-s+1, k-s+1)$ and then consider the case of bidegree $(k-p+1, k-p+1)$. The current $[I^-]$ denotes the current of integration on the regular part of $I^-$. It is not difficult to prove the following proposition:
\begin{proposition}
	Suppose that $I^+_\infty\cap I^-_\infty=\emptyset$. For all $i=0, 1, 2, \cdots$, the currents $\Lambda_{k-s+1}^i([I^-])$ are well-defined positive closed currents of bidegree $(k-s+1, k-s+1)$ and they have the same mass as $[I^-]$ does. Also, we have $\supp \Lambda_{k-s+1}^i([I^-])\subseteq f_+^i(I^-)$.
\end{proposition}

%\begin{proof}
%	We use an inductive argument. in the case of $i=0$, it is obvious. Assume that $\Lambda_{k-s+1}^i([I^-])$ is a current satisfying the statement. Since $\supp \Lambda_{k-s+1}^{i-1}([I^-])\subset I^-_\infty$ and $I^+_\infty\cap I^-_\infty=\emptyset$, $\Lambda_{k-s+1}(\Lambda_{k-s+1}^{i-1}([I^-]))$ is well defined. Proposition \ref{prop:iteration} implies that $\Lambda_{k-s+1}^i([I^-])$ is well-defined and $\supp \Lambda_{k-s+1}^i([I^-])$. Let $\varphi$ be a smooth form of bidegree $(s-1, s-1)$ on $\P^k$. Then, $(f_+^i)^*\varphi$ is an $L^1$-form with bidegree $(s-1, s-1)$ which is smooth outside $I^+_\infty$. So, the following definition makes sense:
%	\begin{displaymath}
%		\langle (f_+^i)_*([I^-]), \varphi\rangle:=\langle [I^-], (f_+^i)^*\varphi\rangle.
%	\end{displaymath}
%	Hence, $(f_+^i)^{-1}(\supp (f_+^i)_*([I^-]))=I^-$ and so $\supp \Lambda_{k-s+1}^i([I^-])\subseteq f_+^i(I^-)$. From the definition, by taking $\varphi=\omega^{s-1}$, we see that the argument about the mass is true.
%\end{proof}

Consider the following sequence of currents:
\begin{displaymath}
	R_n:={(n+1)}^{-1}\sum_{i=0}^n\Lambda_{k-s+1}^i([I^-])
\end{displaymath}

Then, the sequence $\{R_n\}$ has bounded mass. So, there exists a convergent subsequence $\{R_{n_j}\}$ in the sense of currents. Let $R^s_\infty$ denote one of its limit currents.
\begin{proposition}\label{prop:invariant}
	\begin{displaymath}
		\Lambda_{k-s+1}(R^s_\infty)=R^s_\infty,\quad \supp R^s_\infty \subset \overline{I^-_\infty}.
	\end{displaymath}
\end{proposition}

\begin{proof}
	The first part is clear from $\supp R_{n_j}\subset I^-_\infty$ for all $j\in \N$. For a convergent subsequence $\{R_{n_j}\}$, we have
	\begin{displaymath}
		\Lambda_{k-s+1}(R_{n_j})-R_{n_j}={(n_j+1)}^{-1}(\Lambda_{k-s+1}^{n_j+1}([I^-])-[I^-])
	\end{displaymath}
	as $j\to\infty$. Since $\Lambda_{k-s+1}$ is continuous for currents in $\Cc_{k-s+1}$ smooth near $I^+$ and the mass of $\Lambda_{k-s+1}^{n_j}$ is bounded in $j\in\N$, we see that $\Lambda_{k-s+1}(R^s_\infty)=R^s_\infty$ by letting $j\to\infty$.
\end{proof}

For $0<p\leq s$, from Proposition \ref{prop:Holder}, a super-potential $\Uc_{T_+^{s-p}}$ of $T_+^{s-p}$ is H\"older continuous in $V$. Since $\supp R^s_\infty\subset I^-_\infty\Subset V$, the current $R^p_\infty:=(T_+^{s-p})\wedge R^s_\infty$ is well-defined.

\begin{proposition} Let $0<p\leq s$. Assume the existence of the neighborhood $V$ of $I^-_\infty$ in Theorem \ref{thm:main}. Then, we have
	\begin{align*}
		\Lambda_{k-p+1}(R^p_\infty)=R^p_\infty.
	\end{align*}
\end{proposition}

\begin{proof}
	By use of the standard regularization and the Hartog's convergence, we may assume that $T_+^{p-s}$ is smooth. Note that, $R^s_\infty$ has support in $V$. Then, we have
	\begin{align*}
		\Lambda_{k-p+1}(R_\infty^p)=d^{-(p-1)}(f_+)_*((T_+^{s-p})\wedge R^s_\infty)&=d^{-(p-1)}(f_+)_*(d^{-(s-p)}(f_+)^*(T_+^{s-p})\wedge R^s_\infty)\\
		&=d^{-(s-1)}T_+^{s-p}\wedge (f_+)_*R^s_\infty=T_+^{s-p}\wedge R^s_\infty=R^p_\infty.
	\end{align*}
\end{proof}

We also obtain the following corollary from Proposition \ref{prop:Holder}.
\begin{corollary}
	Let $0<p\leq s$. Assume the existence of the neighborhood $V$ of $I^-_\infty$ in Theorem \ref{thm:main}, then the value $\Uc_{T_+^p}(R^p_\infty)$ is finite.
\end{corollary}

The argument in p.53 of \cite{DV} works for any $0<p\leq s$. So, we obtain
\begin{proposition}\label{prop:conv_above}The sequence
	\begin{displaymath}
		d^{-n}\Uc_{T_+^p}\circ \Lambda_{k-p+1}^n
	\end{displaymath}
	goes to $0$ on smooth forms in $\Cc_{k-p+1}$.
\end{proposition}

\section{Proof of Theorem \ref{thm:main} and Theorem \ref{thm:cor}}\label{sec:proof}
For the rest of the note, the current $S\in\Cc_p$ denotes the current in Theorem \ref{thm:main}.
For the proof of Theorem \ref{thm:main}, we set up environment as in \cite{DS09} and \cite{Ahn18}. For simplicity, we will write $L$ and $\Lambda$ for $L_p$ and $\Lambda_{k-p+1}$, respectively. Also, $S_n:=L^n(S)=d^{-pn}(f_+^n)^*(S)$ and $T^p:=T^p_+$.

\begin{definition}
	For $S\in\Cc_p$, we define the dynamical super-potential $\Vc_S$ by
	\begin{align*}
		\Vc_S:=\Uc_S-\Uc_{T^p}-c_S,\quad\textrm{ where }c_S:=\Uc_S(R_\infty)-\Uc_{T^p}(R_\infty)
	\end{align*}
	and the dynamical Green quasi-potential of $S$ by
	\begin{align*}
		V_S:=U_S-U_{T^s}-(m_S-m_{T^s}+c_S)\omega^{p-1}
	\end{align*}
	where $U_S$, $U_{T^p}$ are the Green quasi-potentials of $S$, $T^p$, and $m_S$, $m_{T^p}$ are their mean, respectively.
\end{definition}

The lemma below can be proved in the same way as in Lemma 5.5.5 in \cite{DS09}.
\begin{lemma}[See Lemma 5.5.5 in \cite{DS09}]\label{lem:dynprop}
	\begin{enumerate}
		\item $\Vc_S(R^p_\infty)=0$,
		\item $\Vc_S(R)=\langle V_S, R\rangle$ for smooth $R\in\Cc_{k-p+1}$ and
		\item $\Vc_{L(S)}=d^{-1}\Vc_S\circ \Lambda$ for currents in $\Cc_{k-p+1}$ smooth near $I^+_\infty$.
	\end{enumerate}
\end{lemma}

Different from the case of Theorem \ref{thm:regpoly}, since we cannot say that a super-potential $\Uc_{T^p}$ of $T^p$ is continuous in an open subset $W\Subset \P^k\setminus I^+_\infty$, we cannot say that $\Uc_S-\Vc_S$ is bounded from above on $\Cc_{k-p+1}(W)$ by a constant independent of $S$ in general. However, we have Proposition \ref{prop:conv_above} as an alternative for this.
\medskip

We further introduce some more notations. Let $R\in\Cc_{k-p+1}$ be a smooth current. We choose and fix a constant $\lambda$ such that $1<\lambda<d$ throughout the proof. Let $\eta_n:=\min\{\eta, -\lambda^n\}+\lambda^n$ where $\eta$ is a DSH function in Proposition \ref{prop:greenk}. Then, the DSH-norm of $\eta_n$ is bounded in dependent of $n$ and $K\geq \eta \Theta\geq \eta_n\Theta -\lambda^n\Theta$ where $\eta$ and $\Theta$ are a function and a current in Proposition \ref{prop:greenk}.

For a positive or negative current $S'$ and each $n\in\N$, define  
\begin{displaymath}
	U'_{n, S'}:=\int_{\zeta\neq z}\lambda^n\Theta(\zeta, z)\wedge S'(\zeta)\textrm{, }\quad U''_{n, S'}:=\int_{\zeta\neq z}\eta_n(\zeta, z)\Theta(\zeta, z)\wedge S'(\zeta)
\end{displaymath}
and
\begin{displaymath}
	U'_{\Lambda^n(R)}:=\int_{\zeta\neq z}\lambda^n\Theta(\zeta, z)\wedge \Lambda^n(R)(\zeta)\textrm{, }\quad U''_{\Lambda^n(R)}:=\int_{\zeta\neq z}\eta_n(\zeta, z)\Theta(\zeta, z)\wedge \Lambda^n(R)(\zeta).
\end{displaymath}
Note that if $S'$ is closed, then, $U'_{n, S'}$ is closed and its mass is $c_m\lambda^n\|S'\|$ for a constant $c_m>0$ which is independent of $n$ and $S'$.
\medskip

For the proof of Theorem \ref{thm:main}, we need the following estimate. Observe that Lemma 2.3.9 in \cite{DS09} does not need closedness of the current. Its proof consists of disintegration and singularity estimate. so, we have
\begin{lemma}
	[Lemma 2.3.9 in \cite{DS09}]\label{lem:239} Let $S'$ be a positive current of bidegree $(p, p)$ with bounded mass. Then, we have
	\begin{displaymath}
		\left|\int U''_{n, S'}\wedge \omega^{k-p+1}\right|\lesssim e^{-\lambda^n}\|S'\|.
	\end{displaymath}
The inequality is up to a constant multiple independent of $n$ and $S'$.
\end{lemma}

Now, we start to prove Theorem \ref{thm:main}. It is a direct consequence of the following proposition.
\begin{proposition}\label{prop:main}
	Assume the hypotheses in Theorem \ref{thm:main}. Let $R\in \Cc_{k-p+1}$ be a smooth current. Then, we have
	\begin{displaymath}
		\Vc_{S_n}(R)=d^{-n}\Vc_S(\Lambda^n(R)) \rightarrow 0
	\end{displaymath}
	in the sense of currents.
\end{proposition}

%The point of this proposition is to use Lemma 3.2.10 in \cite{DS09} which is a version of exponential estimate.
%\medskip
We begin with an estimate near $I^-_\infty$ in Lemma \ref{lem:W_3_nbhd_I-}.
\begin{lemma}
	There exist open subsets $W_3\Subset W_2\Subset W_1\Subset W_0\Subset V$ such that $f_+(W_i)\Subset W_i$.
\end{lemma}

\begin{proof}
	Note that $f_+$ is holomorphic outside $I^+$. Since $V\cap I^+=\emptyset$, $f_+(\overline V)$ is compact in $V$. So, simply take $f_+(\overline V)\Subset W_3\Subset W_2\Subset W_1\Subset W_0\Subset V$.	
\end{proof}

%\begin{remark}\label{rmk:derivative}
%	Since $f_+(W_i)\Subset W_i$, we have $W_i\subset (f_+)^{-1}(W_i)\subset (f_+)^{-1}(W_i)$ and hence, $\P^k\setminus (f_+)^{-1}(W_i)\subset \P^k\setminus W_i$. The restriction $f_+:\P^k\setminus (V^+\cup (f_+)^{-1}(W_i)) \to U^+\setminus W_i$ is holomorphic and injective. So, $\|Df\|_{\P^k\setminus (V^+\cup (f_+)^{-1}(W_i))}$ is bounded from below. For a later use, let $m_f>0$ be a lower bound of $\|Df\|_{\P^k\setminus (V^+\cup (f_+)^{-1}(W_3))}$.
%\end{remark}

Let $\chi:\P^k\to [0, 1]$ be a cut-off function such that $\chi\equiv 1$ on $W_1$ and $\supp \chi \Subset W_0$. Let $M>1$ be a constant such that $\|Df_-\|_{\P^k\setminus W_3}<M$. Here, $\|\cdot\|_{\P^k\setminus W_3}$ denotes the uniform norm of coefficients on $\P^k\setminus W_3$ with respect to a fixed finite atlas of $\P^k$.

\begin{lemma}\label{lem:dsh_ddc}
	Let $R\in\Cc_{k-p+1}$ be smooth outside $I^-_\infty$. Then, there exists a constant $c>0$ independent of $R$ and $n$ such that
	\begin{displaymath}
		\|dd^c(\chi U_{\Lambda^n(R)})\|_*\leq cM^{3kn}\|R\|_{\P^k\setminus W_3}
	\end{displaymath}
where $U_{(\cdot)}$ denotes the Green quasi-potential of a given current.
\end{lemma}

\begin{proof}
	We can write 
	\begin{displaymath}
		dd^c(\chi U_{\Lambda^n(R)})=dd^c\chi\wedge U_{\Lambda^n(R)}+d\chi\wedge d^cU_{\Lambda^n(R)}+dU_{\Lambda^n(R)}\wedge d^c\chi +\chi(\Lambda^n(R)-\omega^{k-p+1}).
	\end{displaymath}
	The last term is bounded below by $\omega^{k-p+1}$. Since the first three terms are all smooth since $\Lambda^n(R)$ is smooth outside $I^-_\infty$. They are all bounded by the $C_1$-norm of $U_{\Lambda^n(R)}$ on the support of $d\chi$ or $d^c\chi$ which is compact outside $W_3$. Hence, by Proposition \ref{prop:f+f-smooth}, $\|\Lambda^n(R)\|_{\P^k\setminus W_3}\leq c_1M^{3kn}\|R\|_{\P^k\setminus W_3}$ for some $c_1>0$. Then, due to Lemma 2.3.5 in \cite{DS09}, we get the desired estimate.
\end{proof}

Since the above estimate only depends on the uniform norm of $\Lambda^n(R)$ on the support of $d\chi$ and $d^c\chi$, we see that there exists a constant $\delta_0>0$, which only depend on the distance between $W_3$ and $\P^k\setminus W_1$, such that the same estimate as in Lemma \ref{lem:dsh_ddc} holds for all $\delta>0$ with $\delta<\delta_0$ and for all $n\in\N$: 
\begin{displaymath}
	\|dd^c(\chi U_{(\Lambda^n(R))_\delta})\|_*\leq cM^{3kn}\|R\|_{\P^k\setminus W_3}.
\end{displaymath}

So, we have $c^{-1}M^{-3kn}\|R\|_{\P^k\setminus W_3}^{-1}dd^c(\chi U_{(\Lambda^n(R))_\delta})\in \widetilde D^0_{k-p+1}(W_1)$ for $\delta>0$ with $|\delta|<\delta_0$.

\begin{lemma}\label{lem:delta_dist}
	For $0<\delta<\delta_0$, we have
	\begin{align*}
		\|dd^c(\chi U_{\Lambda^n(R)})-dd^c(\chi U_{(\Lambda^n(R))_\delta})\|_{-2}\lesssim \delta.
	\end{align*}
The inequality is up to a constant multiple independent of $n$ and $\delta$.

\end{lemma}

\begin{proof}
	Let $\varphi$ be a smooth test form with $\|\varphi\|_{C^2}\leq 1$. We have
	\begin{align*}
		&\langle dd^c(\chi U_{\Lambda^n(R)})-dd^c(\chi U_{(\Lambda^n(R))_\delta}), \varphi\rangle=\langle U_{\Lambda^n(R)}- U_{(\Lambda^n(R))_\delta}, \chi dd^c\varphi\rangle\\
		&=\langle \Lambda^n(R)-(\Lambda^n(R))_\delta, U_{\chi dd^c\varphi}\rangle\leq \langle \Lambda^n(R), U_{\chi dd^c\varphi}-(U_{\chi dd^c\varphi})_\delta\rangle
	\end{align*}
	Hence, by Lemma 2.3.5 in \cite{DS09}, we have
	\begin{align*}
		|\langle dd^c(\chi U_{\Lambda^n(R)})-dd^c(\chi U_{(\Lambda^n(R))_\delta}), \varphi\rangle|\lesssim \|U_{\chi dd^c\varphi}\|_{C^1}\delta\lesssim \|\chi dd^c\varphi\|_\infty\delta\lesssim \delta.
	\end{align*}
\end{proof}

\begin{lemma}\label{lem:main_estimate} Let $S_\theta$ be a standard regularization of $S$ for sufficiently small $0<|\theta|\ll 1$. For $0<\delta<\delta_0$, we have
	\begin{align*}
		\left|\int U_{S_\theta}\wedge dd^c(\chi U_{(\Lambda^n(R))_\delta})\right|\lesssim \delta^{-2k^2-4k-2}e^{-\lambda^n}+\lambda^n.
	\end{align*}
Here, the inequality is independent of $\theta$, $\delta$ and $n$.
\end{lemma}
Indeed, the $\theta$ is chosen so that $\supp (dd^c\chi\wedge U'_{n, (\Lambda^n(R))_\delta})_\theta\Subset V$. This condition is completely determined by the function $\chi$. 
\begin{proof}
	We have
	\begin{align*}
		&\int U_{S_\theta}\wedge dd^c(\chi U_{(\Lambda^n(R))_\delta})=\int S_\theta\wedge \chi U_{(\Lambda^n(R))_\delta}-\int \chi\omega^p\wedge U_{(\Lambda^n(R))_\delta}\\
		&\geq \int S_\theta\wedge \chi U'_{n, (\Lambda^n(R))_\delta} + \int S_\theta\wedge \chi U''_{n, (\Lambda^n(R))_\delta}-\int \chi\omega^p\wedge U_{(\Lambda^n(R))_\delta}
	\end{align*}

	We estimate the first integral $\int S_\theta\wedge \chi U'_{n, (\Lambda^n(R))_\delta}$. Note that $U'_{n, (\Lambda^n(R))_\delta}$ is closed and its mass is a constant multiple of $\lambda^n$.
	\begin{align*}
		\int S_\theta\wedge \chi U'_{n, (\Lambda^n(R))_\delta}=\int U_{S_\theta}\wedge dd^c\chi\wedge U'_{n, (\Lambda^n(R))_\delta}+\int \omega^s\wedge \chi U'_{n, (\Lambda^n(R))_\delta}\gtrsim -\lambda^n.
	\end{align*}
	Since $\supp\, dd^c\chi\wedge U'_{n, (\Lambda^n(R))_\delta}\Subset W_0$ and $\|dd^c\chi\wedge U'_{n, (\Lambda^n(R))_\delta}\|_*\lesssim \lambda^n$, the first integral is estimated from the boundedness of $\Uc_{S_\theta}$ in $W_0\Subset V$. The second integral is bounded by the mass of $U'_{n, (\Lambda^n(R))_\delta}$. So, we get the last inequality.
	\medskip
	
	We estimate the second integral $\int S_\theta\wedge \chi U''_{(\Lambda^n(R))_\delta}$. From the negativity of $\eta_n$ and the positivity of $\Theta$, we have
	\begin{align*}
		\int S_\theta\wedge \chi U''_{(\Lambda^n(R))_\delta}&=\int \chi(z)S_\theta(z)\wedge \eta_n(z, \zeta)\wedge \Theta(z, \zeta)\wedge (\Lambda^n(R))_\delta(\zeta)\\
		&\gtrsim \delta^{-2k^2-4k-2}\int_{\P^k\times\P^k} \chi(z)S_\theta(z)\wedge \eta_n(z, \zeta)\wedge \Theta(z, \zeta)\wedge \omega^{k-s+1}(\zeta)\\
		&=\delta^{-2k^2-4k-2}\int U''_{n, \chi S_\theta}\wedge \omega^{k-p+1}\gtrsim -\delta^{-2k^2-4k-2}e^{-\lambda^n}.
	\end{align*}
	The last inequality is from Lemma \ref{lem:239}.
\end{proof}

From the hypothesis on $S\in\Cc_p$ in Theorem \ref{thm:main}, let $\alpha>0$ and $C_\alpha>0$ be two constants such that for all $\theta\in\C$ with sufficiently small $|\theta|$ as in Lemma \ref{lem:main_estimate}, $|\Uc_{S_\theta}(R)-\Uc_{S_\theta}(R')|\leq C_\alpha(\|R-R'\|_{-2})^\alpha$ for $R, R'\in \widetilde D^0_{k-p+1}(W_0)$.

\begin{lemma}\label{lem:nbhd_I-}
	Let $R\in\Cc_{k-p+1}$ be a current smooth outside $I^-_\infty$. Let $S_\theta$ be a standard regularization of $S$ for sufficiently small $|\theta|$ as in Lemma \ref{lem:main_estimate}. We have
	\begin{displaymath}
		\int_{W_1}S_\theta\wedge U_{\Lambda^n(R)}\gtrsim -\lambda^n
	\end{displaymath}
	for all sufficiently large $n$. The inequality is independent of $\theta$ and $n$.
\end{lemma}

\begin{proof}
	From the negativity of the Green quasi-potential, we have
	\begin{align*}
		\int_{W_1}S_\theta\wedge U_{\Lambda^n(R)}&\geq \int \chi S_\theta\wedge U_{\Lambda^n(R)}\\
		&=\int U_{S_\theta}\wedge dd^c(\chi U_{\Lambda^n(R)})+\int \chi \omega^p\wedge U_{\Lambda^n(R)}
	\end{align*}
	Let $\delta>0$ be a small constant to be determined later. Then, the last quantity can be written as
	\begin{align*}
		\int U_{S_\theta}\wedge dd^c(\chi U_{(\Lambda^n(R))_\delta})+\int U_{S_\theta}\wedge (dd^c(\chi U_{\Lambda^n(R)})-dd^c(\chi U_{(\Lambda^n(R))_\delta}))+\int \chi \omega^p\wedge U_{\Lambda^n(R)}
	\end{align*}
	From the H\"older continuity of $\Uc_{S_\theta}$ in $W_0$ with Lemma \ref{lem:dsh_ddc} and Lemma \ref{lem:delta_dist}, the second integral can be estimated as below:
	\begin{align*}
		\left|\int U_{S_\theta}\wedge (dd^c(\chi U_{\Lambda^n(R)})-dd^c(\chi U_{(\Lambda^n(R))_\delta}))\right|\lesssim C_\alpha cM^{3kn}\|R\|_{\P^k\setminus W_3}\left(\frac{\delta}{cM^{3kn}\|R\|_{\P^k\setminus W_3}}\right)^\alpha. 
	\end{align*}

	Since the mass of quasi-potential is uniformly bounded, the third integral is uniformly bounded. From Lemma \ref{lem:main_estimate}, the first integral can be approximated by
	\begin{align*}
		\int U_{S_\theta}\wedge dd^c(\chi U_{(\Lambda^n(R))_\delta})\gtrsim -\delta^{-2k^2-4k-2}e^{-\lambda^n}-\lambda^n.
	\end{align*}

	Altogether, if we choose $\delta=1/(2M^{3k})^{n/\alpha}$, we have
	\begin{align*}
		\int_{W_1}S_\theta\wedge U_{\Lambda^n(R)}\gtrsim -\lambda^n
	\end{align*}
	for all sufficiently larget $n$.
\end{proof}

\begin{lemma}\label{lem:W_3_nbhd_I-}
	Let $R\in\Cc_{k-p+1}$ be a current smooth outside $I^-_\infty$. Let $S_\theta$ be a standard regularization of $S$ for sufficiently small $|\theta|$ as in Lemma \ref{lem:main_estimate}. We have
	\begin{displaymath}
		\int_{W_2}U_{S_\theta}\wedge \Lambda^n(R)\gtrsim -\lambda^n
	\end{displaymath}
	for all sufficiently larget $n$. The inequality is independent of $\theta$ and $n$.
\end{lemma}

\begin{proof}
	\begin{align*}
		&\int_{W_2}U_{S_\theta}\wedge \Lambda^n(R)=\int_{z\in W_2}\int_{\zeta\neq z} S_\theta(\zeta) \wedge K(z, \zeta) \wedge \Lambda^n(R)(z)\\
		&=\int_{z\in W_2}\int_{\zeta\in W_1\setminus \{z\}} S_\theta(\zeta)\wedge K(z, \zeta)\wedge \Lambda^n(R)(z)\\
		&\quad + \int_{z\in W_2}\int_{\zeta\in \P^k\setminus W_1} S_\theta(\zeta)\wedge K(z, \zeta)\wedge \Lambda^n(R)(z)
	\end{align*}
	From the estimate of $K$ in Proposition \ref{prop:greenk}, the second integral is bounded by a constant independent of $\theta$ and $n$. From the negativity of the Green quasi-potential, the first integral is bounded by
	\begin{align*}
		\int_{W_1}S_\theta\wedge U_{\Lambda^n(R)}.
	\end{align*}
	Hence, by Lemma \ref{lem:nbhd_I-}, we get the estimate.
\end{proof}

\begin{proof}[Proof of Proposition \ref{prop:main}] Let $R\in\Cc_{k-p+1}$ be a smooth current.
	By Lemma \ref*{lem:dynprop}, we can write
	\begin{align*}
		\Vc_{S_n}(R)=d^{-n}\Vc_S(\Lambda^n(R)).
	\end{align*}
	By the definition, we have $\Vc_S(\Lambda^n(R))=\Uc_S(\Lambda^n(R))-\Uc_{T^s}(\Lambda^n(R))-c_S\|\Lambda^n(R)\|$. Since the super-potentials on $\P^k$ are upper semicontinuous on $\Cc_{k-s+1}$ which is compact, $\Uc_S(\Lambda^n(R))$ is bounded from above. So, Proposiiton \ref{prop:conv_above} implies that $\limsup_{n\to 0}d^{-n}\Vc_S(\Lambda^n(R))\leq 0$. So, we only consider the estimate of $d^{-n}\Vc_S(\Lambda^n(R))$ from below.
	\medskip
	
	We consider $\Uc_S(\Lambda^n(R))$. From the definition of the super-potential, we have
	\begin{displaymath}
		\Uc_S(\Lambda^n(R))=\lim_{\theta\to 0}\Uc_{S_\theta}(\Lambda^n(R)).
	\end{displaymath}
	Hence, we estimate $\Uc_{S_\theta}(\Lambda^n(R))$ for $\theta\in\C$ with sufficiently small $|\theta|$.
	\medskip
	
	In the rest of the proof, the inequalities $\lesssim, \gtrsim$ are up to a constant multiple independent of $\theta$ and $n$. 
	\medskip
	
	Let $\varepsilon_n>0$ be a sufficiently small positive number to be determined later and $U_{(\cdot)}$ denotes the Green quasi-potential of a given current with respect to a fixed Green quasi-potential kernel in Proposition \ref{prop:greenk}. Since $S_\theta$ is smooth and has the same mean and mass as $S$ does, we can write
	\begin{align*}
		\Uc_{S_\theta}(\Lambda^n(R))&=\int U_{S_\theta}\wedge \Lambda^n(R)-m_S\|\Lambda^n(R)\|\\
		&=\int U_{S_\theta}\wedge (\Lambda^n(R)-(\Lambda^n(R))_{\varepsilon_n})+\int U_{S_\theta}\wedge (\Lambda^n(R))_{\varepsilon_n}-m_S\|\Lambda^n(R)\|
	\end{align*}
	From the negativity of the Green quasi-potential, we have
	\begin{align*}
		\Uc_{S_\theta}(\Lambda^n(R))&\geq \int_{W_2} U_{S_\theta}\wedge \Lambda^n(R)+\int_{\P^k\setminus {W_2}} U_{S_\theta}\wedge(\Lambda^n(R)-(\Lambda^n(R))_{\varepsilon_n})\\
		&\quad +\int U_{S_\theta}\wedge (\Lambda^n(R))_{\varepsilon_n}-m_S\|\Lambda^n(R)\|.
	\end{align*}
	
	We estimate the first integral. From Lemma \ref{lem:W_3_nbhd_I-}, we have
	\begin{align*}
		\int_{W_2} U_{S_\theta}\wedge \Lambda^n(R)\gtrsim -\lambda^n
	\end{align*}
	for all sufficiently large $n$.
	
	For the second integral, we will use the fact that the mass of the Green quasi-potential is uniformly bounded. Proposition \ref{prop:f+f-smooth} implies that $\Lambda^n(R)$ is smooth in $\P^k\setminus W_3$, and from $\|Df_-\|_{\P^k\setminus W_3}<M$ and $f_+(W_3)\Subset W_3$, we have $\|\Lambda^n(R)\|_{\P^k\setminus W_3}\lesssim M^{3kn}$. So, we have
	\begin{align*}
		\|\Lambda^n(R)-(\Lambda^n(R))_{\varepsilon_n}\|_{\P^k\setminus W_2}\lesssim M^{-3kn}\varepsilon_n
	\end{align*}
	and since the mass of $U_{S_\theta}$ is uniformly bounded, the second integral is 
	\begin{displaymath}
		\int_{\P^k\setminus {W_2}} U_{S_\theta}\wedge(\Lambda^n(R)-(\Lambda^n(R))_{\varepsilon_n})\gtrsim -M^{-3kn}\varepsilon_n
	\end{displaymath}
	For the last integral, thanks to Lemma 3.2.10 in \cite{DS09} and Proposition 2.1.6 in \cite{DS09}, we have 
	\begin{align*}
		\int U_{S_\theta}\wedge (\Lambda^n(R))_{\varepsilon_n}\gtrsim \log \varepsilon_n.
	\end{align*}
	If we choose $\varepsilon_n:=\min\{1/2, M^3/2\}^{kn}$, then we have $\Uc_S(\Lambda^n(R))\gtrsim -\lambda^n$ for all sufficiently large $n$. Since $1<\lambda<d$, we have $\liminf_{n\to 0}d^{-n}\Uc_S(\Lambda^n(R))\geq 0$. Together with Proposition \ref{prop:conv_above} again, we get $\liminf_{n\to 0}d^{-n}\Vc_S(\Lambda^n(R))\geq 0$
\end{proof}

\begin{proof}
	[Proof of Theorem \ref{thm:main}] Let $\varphi$ be a smooth test form of bidegree $(k-p, k-p)$. Since $\varphi$ is smooth, there exists $m_\varphi>0$ such that $m_\varphi\omega^{k-s+1}+dd^c\varphi\geq 0$. Then, we have
	\begin{align*}
		\langle S_n-T^p, \varphi\rangle &= \langle dd^cV_{S_n}, \varphi\rangle=\langle V_{S_n}, dd^c\varphi\rangle\\
		&=\langle V_{S_n}, m_\varphi\omega^{k-p+1}+dd^c\varphi\rangle-\langle V_{S_n}, m_\varphi\omega^{k-p+1}\rangle
	\end{align*}
	If we apply Proposition \ref{prop:main} to both terms, we see the desired convergence.	
\end{proof}

\begin{proof}
	[Proof of Theorem \ref{thm:cor}]Let $0<p\leq s$. Let $H$ be an analytic subset of pure dimension $k-p$ and suppose that $H\cap I^-_\infty=\emptyset$. Let $c_H$ denote the degree of $H$. Let $W$ be a subset of the trapping neiborhood $U$ of $I^-_\infty$ such that $H\cap W=\emptyset$. Then, since $U$ is a trapping neighborhood, there exists an $N$ such that $I^-_\infty\Subset  f_+^N(W)\Subset W$. Indeed, there exists $N$ such that $I^-_\infty\subset f_+^N(U)\Subset W$. Note that $I^+_\infty$, $I^-_\infty$ and $T^p$ remain the same if we replace $f$ by $f^N$.
	\medskip
	
	In the proof of Proposition \ref{prop:main}, by applying the same lemmas and propositions to $(f_+)^N$, $\lambda^N$ and $\Lambda^j(R)$ in place of $f_+$, $\lambda$ and $R$, we obtain $\Uc_{c_H^{-1}[H]}((\Lambda^{N})^n(\Lambda^j(R)))\gtrsim -(\lambda^N)^n$ for each $j=0, \cdots, N-1$. By considering a subsequence, Proposition \ref{prop:conv_above} implies that $d^{-Nn}\Uc_{T^p}((\Lambda^N)^n(\Lambda^j(R)))\to 0$ as $n\to\infty$ for each $j=0, \cdots, N-1$. Applying the same argument in Proposition \ref{prop:main}, we see that for a given smooth $R\in\Cc_{k-p+1}$, $\Vc_{L^{Nn}(c_H^{-1}[H])}(\Lambda^j(R))\to 0$ as $n\to \infty$ holds for each $j=0, \cdots, N-1$, which means $d^{-p(Nn+j)}(f_+^{Nn+j})^*(c_H^{-1}[H])$ converges to $T^p$ for each $j=0, 1, \cdots, N-1$. So, the only limit points of $d^{-pn}(f_+^n)^*(c_H^{-1}[H])$ is $T^p$, which completes the proof.
\end{proof}

\begin{remark}
	In our method, one obstacle against measuring the speed of convergence is lack of the speed of convergence in Proposition \ref{prop:conv_above}. In the case of \cite{DS05}, adapting their notation, the convergence is exponentially fast on $U^+$. In this case, simply H\"older continuity of the super-potential of $T^p$ replaces Proposition \ref{prop:conv_above}. 
\end{remark}

\end{document}